\documentclass{amsart}
\usepackage{mathrsfs}

\usepackage{colortbl}

\theoremstyle{plain}
\newtheorem{thm}{Theorem}[section]
\newtheorem{lemma}[thm]{Lemma}
\newtheorem{cor}[thm]{Corollary}
\newtheorem{prop}[thm]{Proposition}

\theoremstyle{definition}
\newtheorem{definition}[thm]{Definition}
\newtheorem{remark}[thm]{Remark}

\font\sevenrm=cmr7

\newcommand\EE{\mathbb{E}}

\newcommand\NN{\mathbb{N}}

\newcommand\RR{\mathbb{R}}

\newcommand\TT{\mathbb{T}}
\newcommand\ZZ{\mathbb{Z}}

\newcommand{\SE}{{\mathscr E}}

\newcommand{\SL}{{\mathscr L}}

\newcommand{\cH}{\mathcal{H}}
\newcommand{\cE}{\mathcal{E}}

\newcommand{\lng}{\langle}
\newcommand{\rng}{\rangle}

\newcommand{\na}{\nabla}
\newcommand{\oo}{\omega}
\newcommand{\si}{\sigma}

\newcommand{\al}{\alpha}
\newcommand{\vep}{\varepsilon}
\newcommand{\vph}{\varphi}

\newcommand{\de}{\delta}
\newcommand{\tvp}{\Phi}
\newcommand{\tvs}{\Psi}

\newcommand{\bv}{\big\vert\!\big\vert}
\newcommand{\Label}[1]{\label{#1}}

\def\mathpal#1{\mathop{\mathchoice{\text{\rm #1}}%
   {\text{\rm #1}}{\text{\rm #1}}%
   {\text{\rm #1}}}\nolimits}

\def\Div{\mathpal{div}}

\def\id{\mathpal{id}}
\def\trace{\mathpal{tr}}

\def\di{\displaystyle}
\def\f{\frac}
\def\a{\alpha }
\def\b{\beta }
\def\D{\Delta }
\def\d{\delta }
\def\e{\varepsilon }

\def\g{\gamma }

\def\n{\nabla }
\def\Om{\Omega }
\def\om{\omega }

\def\s{\sigma }

\begin{document}

\title[]{Generalized stochastic  flows and applications to incompressible viscous fluids}
\author[A. Antoniouk]{Alexandra Antoniouk} \address{Dep. Nonlinear Analysis
\hfill\break\indent Institute of Mathematics NAS Ukraine\hfill\break\indent
 Tereschchenkivska str, 3\hfill\break\indent
  Kyiv, 01 601 UKRAINE}
\email{antoniouk.a@gmail.com}
\author[M. Arnaudon]{Marc Arnaudon} \address{Institut de Math\'ematiques de Bordeaux \hfill\break\indent CNRS: UMR 5251\hfill\break\indent
  Universit\'e Bordeaux 1 \hfill\break\indent
  F33405 TALENCE Cedex, France}
\email{marc.arnaudon@math.u-bordeaux1.fr}

\author[A. B. Cruzeiro]{Ana Bela Cruzeiro} \address{GFMUL and Dep. de Matem\'atica IST(TUL). \hfill\break\indent
  Av. Rovisco Pais\hfill\break\indent
  1049-001 Lisboa, Portugal}
\email{abcruz@math.ist.utl.pt}


%
%

\begin{abstract}\noindent
 We introduce a notion of generalized stochastic flows on manifolds, that extends to the viscous case the one defined by Brenier for perfect fluids. Their kinetic energy extends the  classical kinetic energy to Brownian flows, defined as the $L^2$ norm of their drift. We prove that there exists a generalized flow which realizes the infimum of the  kinetic energy among all generalized flows  with prescribed initial and final configuration. We also construct generalized flows with prescribed drift and kinetic energy smaller  than the $L^2$ norm of the drift.

The results are actually presented for general $L^q$ norms, thus including not only the Navier-Stokes equations but also other equations such as the porous media. \rm

\end{abstract}

\maketitle
\tableofcontents

%
%

\section{Introduction}\label{Section1}
\setcounter{equation}0

Consider the Euler equations describing the velocity 
of incompressible non viscous (perfect) fluids,

\begin{equation}
 \label{E4}
 {\frac{\partial}{\partial t}} u= -(u.\nabla )u  -\nabla p,~~~~ \Div u=0.
\end{equation}

  V. Arnold showed that \rm the corresponding Lagrangian flows $g(t)$, namely the integral curves for $u$, solutions of $\frac{d}{dt} g(t)(x)=u(t, g(t)(x))$, $g(0)(x)=x$, can be characterized as geodesics on the infinite dimensional "manifold" of measure preserving diffeomorphisms of the underlying configuration space (c.f. \cite{Arnold:66}, \cite{Arnold-Khesin:98}). In particular such solutions $g(t)$ minimize the following energy functional, defined in the time interval $[0,T]$,

\begin{equation}
 \label{E5}
S[g]=\frac{1}{2} \int_0^T \int \left|\frac{d}{dt}g(t)(x)\right|^2 dx
\end{equation}
and the corresponding Euler-Lagrange  equations are precisely equations (1.1).

Ebin and Marsden (\cite{Ebin-Marsden:70})  proved that, given a final condition $g(T)$ with some suitable Sobolev regularity (and, in particular, smooth) lying in a small  neighborhood of the identity,  existence and uniqueness of a local minimal geodesic can be obtained. However, in general, there may be situations where such a geodesic is not defined (c.f.\cite{Shnirelman:85}). The main difficulty lies in that the topology induced by the energy is not strong enough to deal with the regularity of the maps.

  To overcome such difficulties Y. Brenier introduced in  \cite{Brenier:89} \rm the notion of generalized solutions for the   minimal action principle, in the spirit of the Monge-Kantarovich problem. These solutions are probability measures defined on the set of Lagrangian  trajectories. This weaker variational approach
allows to consider measure-preserving maps $g(t)$ with possible splitting and crossing during the evolution.
He proved, in particular, that classical solutions can be regarded as  generalized solutions and that there exist
generalized solutions which do not correspond to classical flows.

In the viscous case, where the velocity obeys the 
 Navier-Stokes system
\begin{equation}
 \label{E1bis}
{\frac{\partial}{\partial t}} u= -(u.\nabla )u +\nu \Delta u  -\nabla p,~~~~ \Div u=0
\end{equation}
with a viscosity coefficient $\nu >0$, it is not so clear how to define a corresponding variational principle. 
Following the initial ideas  in \cite{Nakagomi:81} and \cite{Yasue:83} we have considered a stochastic variational principle defined on stochastic Lagrangian flows. The   kinetic \rm  energy is defined for these flows and we have characterized the stochastic processes which are critical for the energy as processes whose drift satisfies Navier-Stokes equation (c.f. \cite{Cipriano-Cruzeiro:07} for flows living on the flat torus and   \cite{Arnaudon-Cruzeiro:10} for flows in a general compact 
  Riemannian manifold).   Formally, when the viscosity coefficient vanishes, the Lagrangian flows become deterministic and  we are back to the Arnold characterization of Euler Lagrangian paths.\rm 

  We mention that another approach to Lagrangian trajectories for the Navier-Stokes equation as geodesics in a different  geometric framework was considered in \cite{Watanabe:81} and \cite{Gl:97}. There the  approach  is deterministic: what is deformed to pass from Euler to Navier-Stokes is the geometry. Other related approaches are considered in \cite{Constantin-Iyer:08}  and \cite{Eyink:10}.

  Actually these ideas are more general: replacing power~$2$ by some $q\ge 2$ in the energy functional, the stochastic processes on the flat torus which are critical for the energy have a drift satisfying   the weighted porous media type
equation 
  \begin{equation}
 \label{E1ter}
{\frac{\partial}{\partial t}} u= \left(-u.\nabla  +\nu \Delta\right)(\|u\|^{q-2} u)  -\nabla p,~~~~ \Div ~u=0
\end{equation}
(see \cite{DNS} and \cite{DGGW} for these equations).
This was recently proved in ~\cite{Antoniouk-Arnaudon:13}.\rm

     So, in the spirit of Brenier's work, we introduce and study here a concept of generalized Lagrangian flows for the Navier-Stokes problem. We present our results for general $L^q$ norms,
thus including other applications (Navier-Stokes corresponding to the case where $q=2$).
\rm

  The abovementionned difficulties encountered in the Euler case to prove existence of critical paths remain in this setting.

\section{Generalized stochastic flows}\label{Section2}
\setcounter{equation}0

Let us consider any compact oriented Riemannian manifold $M$ of dimension $N\ge 2$   and without boundary. We shall use notation $dx$ for integration with respect to the normalized volume measure on $M$.  Let $\cH$ be a Hilbert space.\rm

\begin{definition}
\Label{Stochflow}
 An  {\it It\^o stochastic flow} $g_{{t}}(x)(\oo )$ on $M$, $x\in M$,  $t\in[0,T]$, $T>0$ is a stochastic process which satisfies
 \begin{itemize}
\item[1.] $g_{{0}}(x)( \oo ) {=} x$;
\item[2.] $g_{{t}}(x)( \oo)$ is a semimartingale which satisfies the It\^o {stochastic} equation:
\begin{equation}\Label{SDE}\ \ 
dg_{{t}}(x) (\oo )= \si (g_{{t}}(x)( \oo )) dW_t +u_t( g_{{t}}(x) (\oo ),\oo) \, dt,
\end{equation}
where $\si\in\Gamma(Hom(\cH, TM))$ is a $C^2$-map satisfying for all $x\in M$ $(\si\si^\ast)(x)=\id_{T_xM}$; $W_t$ is a cylindric Brownian motion in the  Hilbert space $\cH$, and $(t,x,\oo)\to u_t(x,\oo)\in T_xM$ is a time dependent adapted drift with locally bounded variation in $x$. 
\end{itemize}
\end{definition}

  Frequently we shall drop the probability space parameter in the notations. \rm

 Recall that if $P(g(x))_t : T_xM\to T_{g(t)(x)}M$ is the parallel transport along $g_t(x)$, then $$dg_t(x)= P(g(x))_td\left(\int_0^\cdot P(g(x))_s^{-1}\circ dg_s(x)\right)_t.$$ 
  Let us mention that at this stage we do not ask for any further regularity of  $u_t(x,\oo)$. The solution to equation~\eqref{SDE} is not unique in general, but for any given $x\in M$ the process $u_t( g_{{t}}(x),\oo)$  is entirely determined
   by the process $g_{{t}}(x)(\om)$: in a local chart, after removing the term coming from Christoffel symbols, $u_t( g_{{t}}(x),\cdot)$ is the time derivative of the drift of $g_t(x)$.

\begin{definition}
\Label{Incdiffflow}
An {\it incompressible 
 stochastic flow}
$g_t(x)(\om)$ is a stochastic flow such that a.s. $\oo$ for all $t\ge 0$, the map $x\mapsto g_t(x)(\om)$ is a volume preserving diffeomorphism of~$M$.
\end{definition}

\begin{prop}
\Label{trampampam}
Assume that the following properties are satisfied:
\begin{itemize}
\item[(i)] $u_t(x,\oo)$ is jointly continuous in $t$ and $x$ and $C^2$ in $x$ with derivatives uniformly bounded in $\om$;
\item[(ii)] almost surely for all $t\in [0,T]$: $\Div\, u_t(\cdot,\oo) =0$;
\item[(iii)] $\text{tr}\ \na_{\si(\cdot)}\si(\cdot)=0$;
\item[(iv)] for all $v\in\cH$:\  $\text{div}\, \si (v)(\cdot)=0$.
\end{itemize}
Then the It\^o stochastic flow $g_t$ is incompressible, namely  $\oo$ a.s., for all $t\in [0,T]$,  for all $f\in C(M)$
\begin{equation}\Label{int}\ 
\int\limits_M f\big(g_{t}(x)(\oo)\big) \, dx = \int_M f(x)\, dx.
\end{equation}
\end{prop}
  Concerning assumptions i) and (ii) on the regularity of $u$, we refer to  ~\cite{Lee:11} for possible more general conditions (but in this reference $u$ does not depend on~$\oo$).\rm

\begin{proof}
Under the condition of the proposition almost surely for all $t$ the map $x\mapsto g_t(x)(\om)$ is a diffeomorphism of $M$. So for $f\in C(M)$
\begin{equation}\Label{int2}\ 
\int\limits_M f\big(g_{t}(x)(\oo)\big) \, dx =\int\limits_Mf(y)|\det T_yg_t^{-1}(\cdot)(\oo)|\,dy
\end{equation}
where $T_yg_t^{-1}(\cdot)(\oo)$ is the tangent map of $y\mapsto g_t^{-1}(y)(\oo)$. On the other hand
condition (iii) implies that equation~\eqref{SDE} is equivalent to the Stratonovich equation
\begin{equation}
\Label{SDES}
dg_{{t}}(x) (\oo )= \si (g_{{t}}(x)( \oo )) \circ dW_t +u_t( g_{{t}}(x) (\oo ),\oo) \, dt.
\end{equation}
 Derivating this equation with respect to the starting point yields
 \begin{equation}
 \Label{DSDES}
d\left(P(g(x))_t^{-1}T_xg_{{t}}\right) =P(g(x))_t^{-1} \left(\n_{T_xg_t}\si  \circ dW_t +\n_{T_xg_t}u_t\, dt\right)
\end{equation}
where again $P(g(x))_t$ denotes parallel transport along $g_t(x)$. Then using the fact that 
$$
T_{g_t(x)}g_{{t}}^{-1} P(g(x))_t P(g(x))_t^{-1}T_xg_{{t}}\equiv \id_{T_xM}
$$
we get 
\begin{equation}
 \Label{invDSDES}
 \begin{split}
&d\left(T_{g_t(x)}g_{{t}}^{-1} P(g(x))_t\right)\\& =T_{g_t(x)}g_{{t}}^{-1} P(g(x))_t P(g(x))_t^{-1} \left(-\n_{P(g(x))_t}\si  \circ dW_t -\n_{P(g(x))_t}u_t\, dt\right)
\end{split}
\end{equation}
and this yields
\begin{equation}
 \Label{trinvDSDES}
 \begin{split}
&d\left(\det T_{g_t(x)}g_{{t}}^{-1} P(g(x))_t\right)\\& =\det T_{g_t(x)}g_{{t}}^{-1} P(g(x))_t \trace\left\{P(g(x))_t^{-1} \left(-\n_{P(g(x))_t}\si  \circ dW_t -\n_{P(g(x))_t}u_t\, dt\right)\right\}\\
&=\det T_{g_t(x)}g_{{t}}^{-1} P(g(x))_t  \left(-\Div\si(g_t(x)))  \circ dW_t -\Div u_t(g_t(x)))\, dt\right)\\
&=0.
\end{split}
\end{equation}
With the initial condition $\det T_xg_{{0}}^{-1} P(g(x))_0=1$ and the property $\det P(g(x))_t\equiv1$ we conclude that 
\begin{equation}
\Label{det1}
\det T_{g_t(x)}g_t^{-1}(\cdot)(\oo)\equiv 1
\end{equation}
almost surely for all $t$ and $x\in M$. But almost surely the map $x\mapsto g_t(x)(\om)$ is a diffeomorphism of $M$, so almost surely for all $t$ and $y\in M$
\begin{equation}
\Label{det1bis}
\det T_{y}g_t^{-1}(\cdot)(\oo)\equiv 1
\end{equation}
  As a conclusion,
$$
\int\limits_M f\big(g_{t}(x)(\oo)\big) \, dx = \int_M f(y)\, dy.
$$
\end{proof}
\begin{remark}

In the following, when looking for incompressible flows, conditions (iii) and (iv) of Proposition~\ref{trampampam} will be unavoidable.

Our main example concerns $M=\TT=\RR/2\pi\ZZ\times\RR/2\pi\ZZ $ the two dimensional torus, $H$ the Hilbert space of real-valued sequences indexed by $\ZZ^2\oplus \ZZ^2$, $\s$ the map defined by 
\begin{equation}
 \label{Torus1}
\s((k_1,k_2)+(0,0))(\theta)=:A_{(k_1,k_2)}(\theta)=(k_2,-k_1)\cos k\cdot \theta
\end{equation}
and 
\begin{equation}
 \label{Torus2}
\s((0,0)+(k_1,k_2))(\theta)=:B_{(k_1,k_2)}(\theta)=(k_2,-k_1)\sin k\cdot \theta.
\end{equation}

Another example is given by any compact semi-simple Lie group $G$ endowed with the metric given by the opposite of the Killing form. Then one can take $H=T_eG$ where $e$ is the identity of $G$ and for $x\in G$, $\s(x)=TL_{x}$ with $L_x$ the left translation.

In the same spirit many examples can be constructed  by projection on symmetric spaces of compact type.
\end{remark}
\begin{definition}
\Label{Drift} For an It\^o stochastic flow $g_t$ satisfying
$$
dg_t= {\si}_t(g_t) dW_t+  u_t(g_t,\oo) dt,
$$ 
we define the {\it drift} as 
$$
D g_t (\oo):=  u_t(g_t,\oo).
$$
  Let us define for $q>1$  the {\it q-energy functional} {by}
\begin{equation}\Label{Eq}\ 
\cE_q(g):=\f1q\EE\left[\int_0^T\int_{M} \left\|Dg_t(x)(\om)\right\|^q\, dx\, dt\right].
\end{equation}
\end{definition}

When $q=2$  the $2$-energy corresponds to  the kinetic energy in hydrodynamics.\rm

As we remarked before, the drift and the energy are uniquely determined by the stochastic flow $g_t$.

{Now we would like to give some explanation and motivation for further definitions.} For the {It\^o} incompressible stochastic flow $g_{ {t}}(x)$, consider the bilinear map $\Theta_t^g$, which to two elements $\vph,\psi\in L_2(M)$  associates  the process:
\begin{equation}
\Label{Tht}\ 
\Theta_t^g (\vph,\psi) =\int\limits_M \vph(x) \psi\big(g_{{t}}(x)\big) dx.
\end{equation}

\begin{definition}
We call the  {bilinear} map $\Theta_t^g$ {\it a  g-flow associated to  {the It\^o stochastic flow} $g_{ {t}}$}.
\end{definition}
{It is easy to see that} for  {fixed} $\vph,\psi\in C^\infty(M)$  {the flow} $\Theta_t^g$ is a real valued semi-martingale process which satisfies 
\begin{eqnarray}\Label{2.11}\ 
\Theta_t^g(\vph,\psi)&=& (\vph,\psi)_{L_2(M)} 
+\sum\limits_{i=1}^\infty \int\limits_0^t  
\Theta_s^g\big(\vph, 
\lng \na\psi,\si_i\rng\big)\, dW_s^i +\\
\nonumber
&+&\int\limits_0^t \Theta_s^g
\big(\vph,\lng \na \psi ,u_{ {s}}
(\cdot,\oo)\rng\big)\, ds
+\frac{1}{2}\int\limits_0^t \Theta_s^g \big(\vph, 
\Delta \psi\big) \, ds,
\end{eqnarray}
where for fixed orthonormal basis $e_i, i\geq 1$ in $\cH$  {we denote} $W_s^i=\lng W_s,e_i\rng$ and $\si_i = \si(e_i)$.

When $g_t$ is a flow of diffeomorphisms, then
\begin{equation}\label{10}
\Theta^{{g}}_t(\varphi,{\psi})=(\theta^\varphi(t,\cdot),{\psi})_{L^2(M)}
\end{equation}
where $x\mapsto \theta^\varphi(t,x)(\om)$ is the function $M\to \RR$ defined by 
\begin{equation}\label{E11}
\theta^\varphi(t,x)=\varphi\left((g(t)(\cdot)(\om))^{-1}(x)\right).
\end{equation}
Notice that for $A$ a Borelian set in $M$ and $\varphi=1_A$ we have $\theta^\varphi(t,x)=1_{g(t)(A)}(x)$.

   It is not clear how to obtain the existence of critical points for our variational
principles within a space of laws of $g$-flows, which are measures supported on real
semimartingales with, in particular, $g(0)(x) = x$. We therefore introduce the
so-called  \it  generalized flows   which  include as particular cases the {$g$}-flows
 $\Theta_t^g$.
\rm

\begin{definition}
\Label{GF}
We call a bilinear map $\Theta_t (\cdot, \cdot)$  {\it a generalized flow} 
if to each $\vph, \psi \in C^\infty(M)$ it associates a continuous 
semi-martingale $t\mapsto \Theta_t(\vph,\psi)$ defined in a common filtered probability space for all $\vph,\psi$, and which for all $\varphi,
\psi, 
\vph_1, \psi_1 \in C^\infty(M)$ satisfies the following list of properties:
\begin{eqnarray*}
(1)& &
 \Theta_t(\varphi,1)\equiv\int_M\varphi(x)\,dx, \Label{E16}\\
 \nonumber
(2) & &  \Theta_t(1,\psi)\equiv\int_M\psi(x)\,dx, \quad
 \hbox{a.s. for all $t\in [0,T]$},\\
 \nonumber
 (3)& &
\Theta_0(\varphi,\psi)=(\varphi,\psi)_{L^2(M)};\\
 (4)& &\text{for $\vph, \psi \geq 0$ it is true that 
$\Theta_t(\vph,\psi) \ge 0$ a.s.;}  \Label{E17}\ \\
 \\
(5)& &  |\Theta_t(\varphi,\psi)|\le\|\varphi\|_{L^2(M)}\|\psi\|_{L^2(M)},\  \text{a.s. for all $t\in [0,T]$;}\\
(6)& & 
d\left[\Theta(\varphi,\psi),\Theta(\varphi_1,\psi_1)\right]_t=
\sum_{i\ge 1}\Theta_t\big(\varphi, \lng\na\psi, \s_i\rng\big)\cdot
\Theta_t \big(\varphi_1, \lng \na\psi_1, \s_i\rng \big)\,dt.\
\end{eqnarray*}
\end{definition}

{Let us notice that (5) and (6) imply}
\begin{equation}\Label{E25bis}\ 
 d\left[\Theta(\varphi,\psi),\Theta(\varphi,\psi)\right]_t\le \|\varphi\|_{L^2(M)}^2\cdot\left\|\na\psi\right\|_{L^2(M)}^2\,dt.
 \end{equation}

\begin{definition}
\label{drifttildeTheta}
For $\vph,\psi\in C^2(M)$, we introduce
\begin{equation}\Label{WTh}\ 
\widetilde{\Theta}_t (\vph,\psi) = \Theta_t(\vph,\psi) -\frac{1}{2}\int\limits_0^t \Theta_s (\vph,\Delta \psi) \, ds.
\end{equation}
If $\widetilde{\Theta}_t $ has absolutely continuous finite variation part, we denote 
 by  
$D\tilde\Theta(\varphi,\psi)$ its drift in the sense of Definition~\ref{Drift}, i.e. the time derivative of its finite variation part.
\end{definition}

\begin{definition}\Label{KE}\ Let $q>1$. We say that the generalized flow $\Theta_t(\cdot,\cdot)$ has {\it q-finite energy} iff 
\begin{itemize}
\item[(i)]\  the semi-martingale $\widetilde{\Theta}_t (\vph,\psi)$
has absolutely continuous finite variation part; 
\item[(ii)]\  the following functional is finite:
\begin{equation}\Label{Epr}\ 
\cE_q'(\Theta) =\frac{1}{q} \sup\limits_{\vph,\psi,\ell, m}\Bigg\{ \EE\int_0^T \sum\limits_{j=1}^m\left[ \sum\limits_{k=1}^\ell 
\frac{\big(D\widetilde{\Theta}_t(\vph_j,\psi_k)\big)^2}{\Theta_t(\vph_j,1)^\al}\right]^{q/2}dt\Bigg\},
\end{equation}
where $\di \a=\f{2(q-1)}q$, sup is taken on all vectors $\vph = (\vph_1,...,\vph_m)$, $\psi = (\psi_1,...,\psi_\ell)$ for any $m, \ell \geq 1$, such that $\vph_j, \psi_k\in C^\infty (M)$, $\vph_j\geq 0$, $\sum\limits_{j=1}^m\vph_j=1$ and $\psi_k$ are such that for all $v\in TM$: $\sum\limits_{k=1}^\ell \lng\na \psi_k, v\rng^2\leq ||v||^2$.
\end{itemize}
We call the functional $\cE_q'(\Theta)$ the {\it  q-energy functional} of the generalized flow $\Theta_t(\cdot,\cdot)$.
\end{definition}

Let $\eta$ be a probability measure on $M\times M$ with marginals equal to $dx$. In particular it disintegrates as $\eta (dx, dx) = dx\, \eta_x (dy)$.

\begin{definition}
\Label{endpointsconf}
The generalized flow $\Theta_t$ is said to {\it satisfy the  endpoints configuration $\eta$} if 
\begin{equation}
\Label{E15}\ 
 \EE\left[\Theta_T(\varphi,\psi)\right]=\int_{M
\times M}\varphi(x)\psi(y)\, \eta(dx,dy). 
\end{equation}
\end{definition}

\begin{prop}
\Label{checklistL}
Let $g_{{t}}(x)$ be an incompressible  {It\^o} stochastic flow  on $M$, satisfying the requirements of Proposition~\ref{trampampam}, such that for all $x\in M$, $g_{{T}}(x)$ has law $\eta_x$.
Then the  g-flow $\Theta^g_t$ associated to the {It\^o} stochastic flow $g_{{t}}$ by  formula (\ref{Tht}) is a generalized flow with  endpoints configuration~$\eta$.
\end{prop}
\begin{proof}
We need to check the requirements of the Definitions~\ref{GF} and {\ref{endpointsconf}}. The requirements (1) -- (5) of Definition~\ref{GF} are trivially fulfilled. For example, to prove (2) it is sufficient to note that
$$
\Theta_t^g (1,\psi) = \int\limits_M \psi (g_t(x)) \,dx = \int\limits_M \psi (x)\, dx.
$$
Property (6) may be easily calculated from (\ref{2.11}). Finally~\eqref{E15} follows from the fact that $g_T(x)$ has law~$\eta_x$.
\end{proof}
\begin{lemma}{\Label{Lemma4.6}}\ 
For a generalized flow $\Theta_t$, the following estimate holds true:
$$
\EE\int_0^T\,\left\vert D\tilde\Theta_t(\varphi,\psi)\right\vert^qdt
\le 2^qq\, \cE'_q(\Theta)\|\varphi\|_{L^\infty(M)}^q\left\|\na\psi\right\|_{L^\infty(M)}^q.
$$
\end{lemma}
\begin{proof}
We can assume that $\|\vph\|_\infty>0$ otherwise the left hand side is~$0$.
First assume that $\vph\ge 0$. Then $\int_M\vph>0$. So
\begin{align*}
&\EE\int_0^T\,\vert D\tilde\Theta_t(\varphi,\psi)\vert^qdt\\
&= \Vert \vph\Vert^q_\infty\cdot \Vert \na \psi\Vert^q_\infty \cdot \left(\int_M\frac{\vph(x)}{\Vert \vph\Vert_\infty}\,dx\right)^{\a q/2}\cdot
\EE \int\limits_0^T\f{\left\vert D\tilde\Theta\left( \frac{\vph}{\Vert \vph\Vert_\infty},
\frac{\psi}{\Vert\na\psi\Vert_\infty}\right)\right\vert^q}{\left(\int_M\frac{\vph}{\Vert \vph\Vert_\infty}\right)^{\a q/2}}\, dt
\\
&\leq\Vert \vph\Vert^q_\infty\cdot \Vert \na \psi\Vert^q_\infty \cdot q\, \cE'_q(\Theta).
\end{align*}
by definition of $\cE'_q(\Theta)$. 

For general $\vph$ we make the splitting into positive and negative parts: $\vph=\vph_+-\vph_-$. Then we write 
\begin{align*}
\vert D\tilde\Theta_t(\varphi,\psi)\vert^q&=\vert D\tilde\Theta_t(\varphi_+,\psi)-D\tilde\Theta_t(\varphi_-,\psi)\vert^q\\
&\le 2^{q-1}\left(\vert D\tilde\Theta_t(\varphi_+,\psi)\vert^q+\vert D\tilde\Theta_t(\varphi_-,\psi)\vert^q\right)
\end{align*}
then we are left to apply the first part of the proof, using the fact that $\|\vph_+\|_\infty\le\|\vph\|_\infty$ and $\|\vph_-\|_\infty\le \|\vph\|_\infty$.\end{proof}
\begin{thm}\Label{2var}\ Let $q>1$ and $g_t$ be an It\^o stochastic flow.
For the energy functionals $\cE'_q(\Theta^g)$ and $\cE_q(g)$
the following inequality is true
\begin{equation}\Label{Eprg}\ 
\cE'_q(\Theta^g)\leq \cE_q(g).
\end{equation}
\end{thm}
\begin{proof}\ 
If $\cE_q(g)=\infty$ then there is nothing to prove. So we assume that $\cE_q(g)<\infty$.
If we take the generalized flow $\Theta^g_t$, corresponding to the ordinary flow $g$ in sense of (\ref{Tht}), then
$$D\widetilde{\Theta}^{ {g}}_t(\vph_j,\psi_k)=\int\limits_M\vph_j(x)\big\lng u_{ {t}}\big(g_{ {t}}(x)\big), \na\psi_k\big(g_{ {t}}(x)\big)\big\rng_{T_xM} dx=:\int\limits_M \vph_j b_k dx$$
where 
\begin{equation}
\Label{bk}\ 
b_k=b_k(t,x)=\big\lng u_{ {t}}\big(g_{ {t}}(x)\big), \na\psi_k\big(g_{ {t}}(x)\big)\big\rng_{T_xM}.
\end{equation}
In this notations  we rewrite the kinetic energy functional $\cE'_q(\Theta^g)$ in the form:
\begin{eqnarray*}\Label{Eprr}\ 
\cE'_q(\Theta^g)
&=&\frac{1}{q} \sup\limits_{\vph,\psi,\ell, m}\Bigg\{ \EE\int_0^T \sum\limits_{j=1}^m\Big[ \sum\limits_{k=1}^\ell 
\frac{\big(\int\limits_M \vph_j(x)b_k(x)dx\big)^2}{\big(\int\limits_M \vph_j(x) dx\big)^\al}\Big]^{q/2}\Bigg\}
\end{eqnarray*}
and denote by $E_q'(\Theta^g)$ the expression under the supremum. Then
\begin{eqnarray*}
E'_q(\Theta^g)
&\leq& \EE\int_0^T \sum\limits_{j=1}^m\Big[ \sum\limits_{k=1}^\ell 
{\Big(\int\limits_M \vph_j(x) dx\Big)^{(1-\al)}}\int\limits_M \vph_j(x)b^2_k(x)dx\Big]^{q/2}dt.
\end{eqnarray*}
Above we used inequality
\begin{equation}\Label{Hol}\ 
\Big(\int\limits_M \vph_j b_k dx\Big)^r \leq \Big(\int\limits_M\vph_j \, dx\Big)^{r-1} \int\limits_M \vph_j b_k^r dx
\end{equation}
with $r=2$.
Due to the fact that  $\psi_k$ are such that for all $u\in TM$: $\sum\limits_{k=1}^\ell \lng\na \psi_k, u\rng^2={:}\sum\limits_{k=1}^\ell b^2_k\leq ||u||^2$ we have:
\begin{eqnarray*}
E'_q(\Theta^g)&\leq& \EE\int_0^T \sum\limits_{j=1}^m\Big[  
{\Big(\int\limits_M \vph_j dx\Big)^{(1-\al)}}\int\limits_M \vph_j||u||^2 dx\Big]^{q/2}dt\leq\\
&\leq& \EE\int_0^T \sum\limits_{j=1}^m\Big[  
{\Big(\int\limits_M \vph_j dx\Big)^{(1-\al)q/2}}
\Big(\int\limits_M \vph_j dx\Big)^{q/2-1}
\int\limits_M \vph_j||u||^q dx\Big]^{q/2}dt.
\end{eqnarray*}
Above we have also applied the inequality (\ref{Hol}) with $r=q/2$
to the last multiple in brackets: $\int\limits_M\vph_j ||u||^2 dx$. 
Since for $\al= {\frac{2(q-1)}{q}}$ we have $
\frac{(1-\al)q}{2}+\frac{q}{2}-1=0$
we obtain the required estimate (\ref{Eprg}).
\end{proof}

\begin{cor}
\Label{checklistL1}
With the assumptions of Proposition~\ref{checklistL}, if we furthermore assume that the It\^o stochastic flow $g_t$ has finite q-energy $\cE_q(g)$ then the associated  g-flow $\Theta_t^g$ also has finite q-energy $\cE'_q(\Theta^g)$.
\end{cor}

\begin{thm}\Label{Version number 4}\ Let $q>1$ an $g_t$ be an It\^o stochastic flow. Then
$$
\cE_q'(\Theta^g)= \cE_q(g).
$$
\end{thm}
\begin{proof} Because of previous theorem it is sufficient to prove
$$
\cE'_q(\Theta^g)\geq \cE_q(g).
$$
First note that for any fixed $\vep >0$ we may choose functions $\vph^\vep_1,...,\vph^\vep_{m_\vep}$ which satisfy $\sum\limits_{j=1}^m\vph_j (x) =1$ on $M$ such that ${supp} \, (\vph_j^\vep)\subset B(x_j, \vep)$ and for some fixed $\delta$ independent on $\vep$ and $j$:
\begin{equation}
\Label{vep}\ 
\int\limits_M \vph_j^\vep (x)dx \geq \delta\vep^d.
\end{equation}
Using notation~\eqref{bk} and $b=(b_1,\ldots,b_\ell)\in\RR^\ell$ together with $\|b\|=\|b\|_{\RR^\ell}$ the Euclidean norm, we may write
\begin{eqnarray}
\nonumber
\ \ \ \ \ \ 0 &\leq& \cE_q(g) -\cE_q'(\Theta^g) \leq \frac{1}{q}\EE \int\limits_0^T \Bigg\{\int\limits_M 
\bv b(y)\bv^q dy - \sum\limits_{j=1}^m
\frac{\bv\int\limits_M \vph_j(x)b(x)dx\bv^q}{\Big(\int\limits_M \vph_j(x) 
dx\Big)^{q-1}}\Bigg\}dt =\\
\nonumber
&=& \frac{1}{q}\EE \int\limits_0^T\int\limits_M
\sum\limits_{j=1}^m  \vph_j(y)
\Bigg\{ \bv b(y)\bv^q  -  
\frac{\bv\int\limits_M \vph_j(x)b(x)dx\bv^q}{\Big(\int\limits_M \vph_j(x) dx\Big)^{q}}\Bigg\}dy\,dt \leq\\
& =&\EE \int\limits_0^T\int\limits_M
\sum\limits_{j=1}^m  \vph_j(y)
\bv b(y)\bv^{q-2}\big\lng b(y), b(y)-  
\frac{\int\limits_M \vph_j(x)b(x)dx}{\int\limits_M \vph_j(x) dx}\,\big\rng\, dy\,dt.
\Label{A} 
\end{eqnarray}
Above we have used the convexity of the function $f(x) = x^q$ and applied the inequality
$$
f(x)-f(x_0)\leq \lng f'(x),x-x_0\rng.
$$
Conducting the change of variables we continue
\begin{align*}
&(\ref{A})=\EE \int\limits_0^T\!\!\sum\limits_{j=1}^m 
\frac{1}{\int\limits_M \!\!\vph_j dx}\!
\int\limits_M\!\!
\vph_j(y)
\bv b(y)\bv^{q-2}\big\lng b(y), \!\!\int\limits_M \!\!b(y)\vph_j(x)dx\!-\!\!\!
\int\limits_M \!\!\vph_j(x)b(x)dx \big\rng\, dy\,dt=\\
&=\EE \int\limits_0^T\!
\sum\limits_{j=1}^m 
\frac{1}{\int\limits_M 
\vph_j dx}\!\!
\int\limits_{M\times M}\!
\vph_j(y)
\bv b(y)\bv^{q-2}\big\lng b(y), b(y)\vph_j(x)\!-\!
\vph_j(x)b(x) \big\rng\, dx\, dy\,dt=\\
&=\EE \int\limits_0^T
\sum\limits_{j=1}^m 
\frac{1}{\int\limits_M \!
\vph_j dx}\!
\int\limits_{B(0,\vep)}
\int\limits_M\!
\vph_j(y)
\bv b(y)\bv^{q-2}\!
\big\lng b(y), \vph_j(y\!+\!r)\big(b(y)\!-\!
b(y\!+\!r)\big) \big\rng dy\,dr\,dt\leq\\
&\leq\frac{1}{\de\vep^d}\int\limits_{B(0,\vep)}
\sum\limits_{j=1}^m \EE \int\limits_0^T \int\limits_M
\vph_j(y) \bv b(y)\bv^{q-1} \bv b(y) -b(y+r)\bv dy\, dt\, dr\leq\\
&\leq \frac{1}{\de\vep^d}\int\limits_{B(0,\vep)} 
\Bigg( \EE \int\limits_0^T \int\limits_M \bv b(y) \bv^q dy\, dt \Bigg)^{\frac{q-1}{q}}
\cdot \Bigg(\EE\int\limits_0^T \int\limits_M \bv b(y) - b(y+r)\bv^q dy\, dt
\Bigg)^{1/q} dr\leq\\
&\leq \frac{\text{\it Vol}\,(B(0,\vep))}{\de\vep^d} \,q^{\frac{q-1}{q}}\cE (g)^{
\frac{q-1}{q}} \sup\limits_{r\in B(0,\vep)} \bv \tau_r b - b\bv_q,
\end{align*}
where $\bv\cdot\bv$ is a standard $L_q$-norm of vector functions on $M$ and $\tau_a b(y) = b(y-a)$. Above we also used the property (\ref{vep}) of functions $\vph_j$. Due to the continuity in $L_q(M^\ell)$ of the map $a\mapsto \tau_a b$ for $b\in L_q (M^\ell)$, the above expression tends to zero as $\vep \to 0$, that finishes the proof.

\end{proof}
\section{Existence of generalized flow minimizing the energy}\label{Section4}
\setcounter{equation}0

In this section we investigate the conditions under which there exists a generalized flow $\Theta_t$, which minimizes the  energy functional $\cE_q'(\Theta)$ defined in (\ref{Epr}).

\begin{definition}
\Label{Hq}
For $q>1$, let $\cH_q(\eta)=\cH_q(\s,\eta,T)$ the set of laws of incompressible It\^o stochastic flows $g_t$  solutions to~\eqref{SDE}, satisfying the requirements of Proposition~\ref{trampampam}, with finite q-energy $\cE_q(g)$ and endpoints configuration~$\eta$.

We let 
\begin{equation}
\label{E6}
\cE_q(\eta)=\cE_q(\s,\eta, T):=\inf\left\{\cE_q(g),\ \ \hbox{Law}(g) \in\cH_q(\eta) \right\}
\end{equation}
where by convention $\cE_q(\eta)=\infty$ if $\cH_q(\eta)$ is empty.
\end{definition}
\begin{definition}
\Label{Hprimeq}
For $q>1$, let $\cH'_q(\eta)=\cH'_q(\s,\eta,T)$ the set of laws of 
generalized flows
 $\Theta_t$ of Definition~\ref{GF}, with finite q-energy~\eqref{Epr} and 
 endpoints configuration~$\eta$. 

We let 
\begin{equation}
\label{E6gen}
\cE'_q(\eta)=\cE'_q(\s,\eta, T):=\inf\left\{\cE'_q(\Theta),\ \ \hbox{Law}(\Theta) \in \cH'_q(\eta) \right\}
\end{equation}
with by convention $\cE'(\eta)=\infty$ if there is no generalized flow with endpoints configuration~$\eta$.
\end{definition}

\begin{remark}
\Label{A1}
>From Proposition~\ref{checklistL} and Theorem~\ref{Version number 4} we have
\begin{equation}
\Label{EprimeleE}\ 
\cE'_q(\eta)\le \cE_q(\eta).
\end{equation}
\end{remark}

\begin{thm}
\label{T1}
If $\cE'_q(\eta)<\infty$ then there exists a generalized flow $\Theta$ with law in $\cH'_q(\eta)$,  such that 
$\cE'_q(\Theta)=\cE'_q(\eta)$. In other words, the infimum of the q-energy functional for given endpoints configuration~$\eta$ is achieved on an  element of the set $\cH'_q(\eta)$. 
\end{thm}
\begin{proof}


Assume $\cE'_q(\eta)<\infty$. Let $\{\Theta^n\}_{n\ge 1}$ be a sequence of generalized flows with laws in $\cH'_q(\eta)$ with energies converging to the infimum value {for given endpoints configuration $\eta$}:
\begin{equation}
 \Label{22}
\lim_{n\to \infty}\cE'_q(\Theta^n)=\cE'_q(\eta).
\end{equation}
Consider  {two } sequences ${\Phi}=\{{\tvp}^j\}_{j\ge 1}$, 
${\Psi}=\{{\tvs}^k\}_{k\ge 1}$
of 
elements of $C^\infty(M)$, where ${\tvp}$ is dense in the 
topology of uniform convergence, and ${\tvs}$ is  dense in 
the topology of uniform convergence of functions and their first and second 
order derivatives. 

We prove that there exists a subsequence $\{\Theta^{n_\ell}\}_{\ell\ge 1}$ 
of {generalized flows}  { such that the family {of semimartingales} 
$\left\{\Theta_t^{n_\ell}({\tvp}^j,{\tvs}^k)\right\}_{j,k\ge 1}$ converges as $\ell$ goes to infinity 
to some family of semimartingales  $\left\{\Theta_t({\tvp}^j,
{\tvs}^k)\right\}_{j,k\ge 1}$, and all $\Theta_t({\tvp}^j,
{\tvs}^k), \ {j,k\ge 1}$ are defined in the same probability space.}

For this it is sufficient to prove that 
for fixed $j,k\ge 1$ and  {$\Phi^j\in\Phi$},  {$\Psi^k\in\Psi$} the family {of semimartingales} $\left\{\Theta_t^{n}
({\tvp}^j,{\tvs}^k)\right\}_{n\ge 1}$ is tight. A diagonal extraction of 
subsequence then gives the result.

Proving that the family $\left\{\Theta_t^{n}({\tvp}^j,{\tvs}^k)\right\}_{n\ge 1}$ is tight can be done as in~\cite{ACGF} by checking the conditions of Theorem~3 in \cite{Zheng:85}. 
For any fixed $n\in\NN$, define ${\{}Y_i^n{\}}_{i\in \NN}$ as the family of all 
$\tilde\Theta^n({\tvp}^j,{\tvs}^k)$ and $\di 
\f12\int_0^\cdot\Theta_s^n({\tvp}^j,\D {\tvs}^k)\,ds$, for $j,k\ge 1$ and 
all their covariations {enumerated in proper order}. They can be renamed this way because there is a 
countable number of them. 
 {Like in~\cite{ACGF} and due to  Lemma \ref{Lemma4.6} we have:}
\begin{equation}
 \begin{split} \Label{E24}\
\EE\int_0^T\left\vert D\tilde\Theta_t^n({\tvp}^j,{\tvs}^k)\right\vert^q\,dt&\le q2^q \left\|{\tvp}^j\right\|_{L^\infty(M)}^q\left\|\na{\tvs}^k\right\|_{L^\infty(M)}^q\cE'_q(\Theta^n)\\&
\le q2^q \left\|{\tvp}^j\right\|_{L^\infty(M)}^q\left\|\na{\tvs}^k\right\|_{L^\infty(M)}^q(\cE'_q(\eta)+1)
\end{split}
\end{equation}
for $n$ sufficiently large.
Due to (5) in Definition~\ref{GF} we also have 
\begin{equation}
 \Label{E39}\ 
\EE\left[\int_0^T\left|\Theta_s^n({\tvp}^j,\D {\tvs}^k)\right|^2\,ds\right]\le T\,\|{\tvp}^j\|_{L^2(M)}^2\left\|\D{\tvs}^k\right\|_{L^2(M)}^2.
\end{equation}
On the other hand due to~\eqref{E25bis} the  {covariation} of $\tilde\Theta^n_t({\tvp}^j,{\tvs}^k)$ satisfies 
\begin{equation}
 \Label{E25ter}\ 
d\left[\tilde\Theta^n({\tvp}^j,{\tvs}^k),\tilde\Theta^n({\tvp}^j,{\tvs}^k)\right]_t\le \|{\tvp}^j\|_{L^2(M)}^2\left\|\na{\tvs}^k\right\|_{L^2(M)}^2\,dt.
\end{equation}
Therefore the derivative of this covariation is uniformly bounded.

From~\eqref{E24}, \eqref{E39} and \eqref{E25ter}, by  Theorem~3 in 
\cite{Zheng:85}, for any fixed  $d\in \NN$, there exists a subsequence 
$n\mapsto \g_d(n)$ such that the {sequence of} vector{s} $\left\{(Y_i^{\g_d(n)})_{1\le i\le 
d}\right\}_{n\in\NN}$ converges in law to some vector $(Y_i^d)_{1\le i \le 
d}$. By a diagonal extraction together with Kolmogorov extension theorem, 
we can find a subsequence $n\mapsto \g(n)$ and a family of random 
variables $(Y_i)_{{i\in\NN}}$
such that for all $d\in \NN$, the sequence {of finite vectors}
$\left\{(Y_i^{\g(n)})_{1\le i\le d}\right\}_{n\in\NN}$ converges to {the vector}
$(Y_i)_{1\le i \le d}$. For simplicity we rename $Y_i^{\g(n)}$ by $Y_i^n$. 
Denote by 
\begin{equation}
\Label{LimThetan}\ 
\tilde\Theta_{{t}}({\tvp}^j,{\tvs}^k)=\SL\lim_{n\to\infty}\tilde\Theta_{{t}}^n({\tvp}^j,{\tvs}^k)
\end{equation}
 and
 \begin{equation}
 \Label{LimintThetan}\ 
 A_{{t}}({\tvp}^j,{\tvs}^k)=\SL\lim_{n\to\infty}\f12\int_0^{{t}}\Theta_s^n({\tvp}^j,\D {\tvs}^k)\,ds, 
 \end{equation}
 where $\SL\lim$ denotes limit in law.
  Then the required limiting process is
\begin{equation}\Label{E41} \ 
\Theta_t({\tvp}^j,{\tvs}^k):=\tilde\Theta_t({\tvp}^j,{\tvs}^k)+ A_t({\tvp}^j,{\tvs}^k).
\end{equation}

{Now} we are going to prove that the process $\Theta_t$ may be extended to  
generalized flow in the sense of Definition~\ref{GF}, which satisfies the 
endpoints condition {$\eta$} and minimizes the q-energy functional~$\cE'_q(\Theta)$.

{From the procedure above for the construction of the sequence $Y_i^n:=Y^{\gamma(n)}_i$}  we have that {for any fixed $d$} all linear combinations of $Y_i^n$, $1\le i \le d$ converge to the same linear combination of $Y_i$, $1\le i \le d$, since linear combinations are continuous functions of finite families. 
 
As for the starting point we have 
\begin{equation}
 \Label{E29}\ 
\tilde\Theta_0({\tvp}^j,{\tvs}^k)=({\tvp}^j,{\tvs}^k)_{L^2(M)}.
\end{equation}

By Theorem~10 in \cite{Meyer-Zheng:84} (which is stated for $q=2$ but  extends to any $q>1$ using H\"older inequality), we have 
\begin{equation}
 \Label{E26}\ 
\EE\left[\int_0^Tdt\left(D\tilde\Theta_t({\tvp}^j,{\tvs}^k)\right)^{q}\right]\le\limsup_{n\to\infty}\EE\left[\int_0^Tdt\left(D\tilde\Theta_t^n({\tvp}^j,{\tvs}^k)\right)^q\right].
\end{equation}
This by Lemma~\ref{Lemma4.6} implies
\begin{equation}
 \Label{E27}\ 
\EE\left[\int_0^Tdt\left(D\tilde\Theta_t({\tvp}^j,{\tvs}^k)\right)^q\right]\le q2^q\left\|{\tvp}^j\right\|_{L^\infty(M)}^q\left\|\na{\tvs}^k\right\|_{L^\infty(M)}^q(\cE'_q(\eta)+1).
\end{equation}
{We have that all covariations of $Y_i^n$ and $Y_j^n$ converge to 
covariation of $Y_i$, $Y_j$, this is due to Theorem~3 
in~\cite{Zheng:85}.}
So inequality~\eqref{E25ter} also extends to the limiting process:
\begin{equation}
 \Label{E28}\ 
d\left[\tilde\Theta({\tvp}^j,{\tvs}^k),\tilde\Theta({\tvp}^j,{\tvs}^k)\right]_t\le \|{\tvp}^j\|_{L^2(M)}^2\left\|\na{\tvs}^k\right\|_{L^2(M)}^2\,dt.
\end{equation}

Furthermore, by bilinearity of all the $\tilde\Theta^n_t$, statements~\eqref{E29},~\eqref{E27} and~\eqref{E28} are still true with functions $\varphi$ and $\psi$ which are linear combinations of ${\tvp}^j$ and ${\tvs}^k$. Moreover {the limit} $\tilde\Theta_t$ is {also} bilinear for theses combinations.

{Thus by (\ref{E41}) we} have defined $\Theta_t(\Phi^j,\Psi^k)$  for $\{\Phi^j\}_{j\ge 1}$ and 
$\{\Psi^k\}_{k\ge 1}$ dense in $C^\infty(M)$ in corresponding topology. Our 
next task is to define it for any $\vph, \psi\in C^\infty(M)$. Let us emphasize 
the fact that {for} all {$j, k$} the $\Theta_t(\Phi^j,\Psi^k)$ are defined on the same 
filtered probability space which will also serve to define the $\Theta_t(\vph,
\psi)$. 

 For $\varphi,\psi\in C^\infty (M)$ there exist {sub}sequences 
 $\{{\tvp}^{j_\ell}\}_{\ell\ge 1}$ and $\{{\tvs}^{k_\ell}\}_{\ell\ge 1}$ which 
 converge uniformly to $\varphi$ and $\psi$ (for the second sequence 
 uniform convergence holds for functions and their first order 
 derivatives). From \eqref{E29},~\eqref{E27}, and~\eqref{E28}  and the 
 bilinearity of $\tilde\Theta_t$ we deduce that  
 $\tilde\Theta_t({\tvp}^{j_\ell},{\tvs}^{k_\ell})$ converges to a 
 semimartingale $\tilde\Theta(\varphi,\psi)$ which does not depend on 
 the {choice of sub}sequences  $\{{\tvp}^{j_\ell}\}_{\ell\ge 1}$ and 
 $\{{\tvs}^{k_\ell}\}_{\ell\ge 1}$. Here the convergence is taken in the 
 topology of $L^q$ convergence of the drift and the {convergence of the} quadratic variation 
 (the so-called $H_q$ topology). It is also easy  to check that $(\varphi,
 \psi)\mapsto \tilde\Theta_t(\varphi,\psi)$ is bilinear {map} and that  for all 
 $\varphi,\psi\in C^\infty (M)$, $\tilde\Theta_t(\varphi,\psi)$ is the limit in 
 law of $\{\tilde\Theta_t^n(\varphi,\psi)\}_{n\ge 1}$ (the last point is due to 
 the fact that the bounds in the right of~\eqref{E24} 
 and~\eqref{E25ter} can be taken indepent of $n$, and this allows to 
 identify the limit of a subsequence of $\{\tilde\Theta_t^n(\varphi,
 \psi)\}_{n\ge 1}$ {with}
  $\tilde\Theta_t(\varphi,\psi)${)}. Similarly, for 
 $\varphi_1, \psi_1,\varphi_2,\psi_2\in C^\infty(M)$, the process 
 $[\tilde\Theta(\varphi_1,\psi_1),\tilde\Theta(\varphi_2,\psi_2)]_{{t}}$ is the 
 limit in law of $\{[\tilde\Theta^n(\varphi_1,\psi_1),
 \tilde\Theta^n(\varphi_2,\psi_2)]_{{t}}\}_{n\ge 1}$.

From~\eqref{LimintThetan}, by bilinearity {of $\Theta_t^n$} and~\eqref{E39} we have, by the same subsequence procedure, {that} for all $\vph,\psi\in C^\infty(M)$
\begin{equation}
\Label{L2}\ 
A_{{t}}(\vph,\psi)=\SL\lim_{n\to\infty}\f12\int_0^{{t}}\Theta_s^n(\vph,\D\psi).
\end{equation}
This together with~\eqref{LimThetan} yields
\begin{equation}
\Label{L3}\ 
\Theta_{{t}}(\vph,\psi)=\SL\lim_{n\to\infty}\Theta_{{t}}^n(\vph,\psi)
\end{equation}
and in particular
\begin{equation}
\Label{L4}\ 
\Theta_{{t}}(\vph,\D\psi)=\SL\lim_{n\to\infty}\Theta_{{t}}^n(\vph,\D\psi).
\end{equation}
So by integration we get 
\begin{equation}
\Label{L5}\ 
\int_0^{{t}}\Theta_s(\vph,\D\psi)=\SL\lim_{n\to\infty}\int_0^{{t}}\Theta_s^n(\vph,\D\psi).
\end{equation}
Finaly by identifying the limits in~\eqref{L2} and~\eqref{L5} we have
 \begin{equation}
  \Label{E42}\ 
A_t(\vph,\psi)=\f12\int_0^t\Theta_s(\varphi,\D\psi)\,ds.
 \end{equation}
So
 for $\varphi,\psi\in C^\infty(M)$
\begin{equation}\label{E38}
\Theta_t(\varphi,\psi)=\tilde\Theta_t(\varphi,\psi)+\f12\int_0^t\Theta(\varphi,\D\psi)\,ds.
\end{equation}
It remains to prove that $\text{{Law}}
\,\Theta_{{t}}$ is an element of $\cH'(\eta)$ and that 
\begin{equation}\Label{E32}\ 
\cE'_q(\Theta)\le \cE'_q(\eta).
\end{equation} 
By passing to the limit {in} (1)-(6), of Definition~\ref{GF} {for $\Theta_t^n$ we get them for $\Theta_t$} . 

We are left to prove~\eqref{E32}. For this we need to improve~\eqref{E27}.

For $\ell\ge 1$, $\vph, \psi_1,\ldots,\psi_\ell\in C^\infty(M)$,  $\psi=(\psi_1,\ldots,\psi_\ell)$, let 
$$
|D\tilde\Theta_t^n(\vph,\psi)|:=\left(\sum_{k=1}^\ell D\tilde\Theta_t^n(\vph,\psi_k)^2\right)^{1/2}.
$$
 Theorem~10 in \cite{Meyer-Zheng:84} (extended to any {$q>1$}
 says that for any $\varphi,\psi\in C^\infty(M)$ and any $K>0$, if for all $n\ge 1$
$$
\EE\left[\int_0^T\left\vert D\tilde\Theta_t^n(\varphi,\psi)\right\vert^q\,dt\right]\le K
$$
then 
$$
\EE\left[\int_0^T\left\vert D\tilde\Theta_t(\varphi,\psi)\right\vert^q\,dt\right]\le K.
$$
Let 
$$
K=\liminf_{n\to\infty}\EE\left[\int_0^T\left\vert D\tilde\Theta_t^n(\varphi,\psi)\right\vert^q\,dt\right].
$$
Consider a subsequence $\Theta^{n_\ell}$ such that 
$$
\lim_{\ell\to\infty}\EE\left[\int_0^T\left\vert D\tilde\Theta_t^{n_\ell}(\varphi,\psi)\right\vert^q\,dt\right]=K.
$$
Then fixing  $\e>0$ and applying the above result to the sequence $\tilde\Theta^{n_\ell}$ for sufficiently large $\ell$ we obtain
$$
\EE\left[\int_0^T\left\vert D\tilde\Theta_t(\varphi,\psi)\right\vert^q\,dt
\right]\le K+\e.
$$
Letting $\e\to 0$ we obtain
\begin{equation}
 \label{E30}
\EE\left[\int_0^T\left\vert D\tilde\Theta_t(\varphi,\psi)\right\vert^q\,dt\right]\le \liminf_{n\to\infty}\EE\left[\int_0^T\left\vert D\tilde\Theta_t^n(\varphi,\psi)\right\vert^q\,dt\right].
\end{equation}
Letting $\varphi^j, \psi^k$ as in~\eqref{Epr} {with $\psi = (\psi_1,...,\psi_\ell)$}, we have
\begin{align*}
 \sum_{j=1}^m\EE\left[\int_0^T\f{\left\vert D\tilde\Theta_t(\varphi^j,
 \psi)\right\vert^q}{\Theta_t(\varphi^j,1)^\a}\,dt\right]
 &\le 
 \sum_{j=1}^m\liminf_{n\to\infty}\EE\left[\int_0^T\f{\left\vert 
 D\tilde\Theta_t^n(\varphi^j,\psi)\right\vert^q}
 {\Theta_t^n(\varphi^j,1)^\a}\,dt\right]\\
&\le \liminf_{n\to\infty}\sum_{j=1}^m\EE\left[\int_0^T\f{\left\vert 
D\tilde\Theta_t^n(\varphi^j,\psi)\right\vert^q}
{\Theta_t^n(\varphi^j,1)^\a}\,dt\right]\\
&\le \cE'_q(\eta).
\end{align*}
Finally taking the supremum in the left as in~\eqref{Epr} yields
\begin{equation}
 \label{E31}
\cE'_q(\Theta)\le \cE'_q(\eta)
\end{equation}
and this achieves the proof.
\end{proof}

\section{Constructing generalized flows with prescribed drift}\label{Section5}
\setcounter{equation}0
We have shown in section~\ref{Section2} that semimartingale flows can be regarded as generalized ones.    In particular (considering $q=2$)  a smooth
solution of Navier-Stokes equation will thus give rise to a generalized solution of the corresponding variational problem.\rm We shall now show that these are not the only possible generalized  flows: indeed, we can define
weaker solutions (which, in particular, will not necessarily correspond to semimartingale flows) of the Navier-Stokes variational problem built upon weak solutions of some transport equations.

Consider a deterministic drift $ b(t,x)$ such that $b\in L^1([0,T], L^q (TM))$ for some $q>1$, and $\Div b\equiv 0$ in the weak sense.  
The following result states existence of generalized flows with drift~$b$ (see \cite{Figalli:08}, \cite{Flandoli-Gubinelli-Priola:10} and \cite{Lee:11} for related results).
\begin{thm}
\label{T4.1}
There exists a generalized flow $\Theta_t$ with drift $ b(t,x)$ i.e. such that for all $\varphi,\psi\in C^\infty(M)$,  $(t,\om)\in [0,T]\times \Om$ almost everywhere  
\begin{equation}
\label{4.1}
D\tilde \Theta_t(\varphi,\psi)=\Theta_t(\varphi,\Div(\psi b)),
\end{equation}
and with kinetic energy smaller than or equal to $\di\f1q\int_0^T\left(\int_M\|b(t,x)\|^q\,dx\right)\,dt$.
\end{thm}
\begin{proof}
For $\e>0$, as in \cite{Lee:11} Section~4.4 we regularize $ b$ by using the de Rham-Hodge semi-group on differential forms $e^{\e\square}$ with $\square = -(d\d+\d d)$, $\d$ the codifferential form of~$d$. For a differential $1$-form $\a$ on $M$ we denote by $\a^\sharp$ the vector field on $M$ associated to $\a$ by the metric, and for a vector field $A$ on $M$ we denote by $A^\flat$ the differential $1$-form associated to $A$ by the metric: we have 
$$
\langle \a, A\rangle=\langle \a^\sharp, A\rangle=\langle\a,A^\flat\rangle
$$ 
where the first bracket is for duality, the second one is the scalar product in $TM$, the third one the scalar product in $T^\ast M$.
By letting 
\begin{equation}\label{E64}
 (b')^\e=\left(e^{\e\square}( b^\flat)\right)^\sharp
\end{equation}
we get a smooth time-dependent vector field satisfying $\Div  (b')^\e=0$ (see \cite{Lee:11} Proposition~4.4.1). Then we regularize $(b')^\e$ in time by convolution with a smooth kernel with support $[-\e/2,\e/2]$ (for this we need to extend $b$ by letting $b(t,x)=0$ for $t<0$ and for $t>T$). Let us call $b^\e(t,x)$ the regularized vector field.  It is also divergence free since the divergence operator commutes with the time integration,
and it approximates    $ b$ in $( L^q([0,T]\times M, TM))$.
For each $\e>0$ we can construct a semimartingale flow as a strong solution to \eqref{SDE} where $u(t,g(t)(x),\om)$ has been replaced by $b^\e(t,g(t)(x))$. Let us denote by $g^\e$ the solution. Letting $(\e_n)_{n\ge 0}$ a sequence of positive numbers decreasing to $0$ we let $\Theta^n=\Theta^{g^{\e_n}}$. Now using the fact that
$$
\SE(g^\e)=
\f1q\int_0^T\left(\int_M\|b^\e(t,x)\,dt\|^q\,dx\right)\,dt\le \f1q\int_0^T\left(\int_M\|b(t,x)\,dt\|^q\,dx\right)\,dt+1
$$
for $\e$ sufficiently small
 we proceed similarly to the proof of Theorem~\ref{T1} { to establish that possibly by 
 extracting a subsequence,}
 there exists a generalized flow $\Theta_t$ such that for all 
 $\varphi,\psi\in C^\infty(M)$ $\Theta^n(\varphi,\psi)$ converges in law to  $\Theta(\varphi,
 \psi)$ and  $\int_0^\cdot D\tilde \Theta_s^n(\varphi,\psi)\,ds$ converges in law to the drift 
 of $\tilde \Theta(\varphi,\psi)$. We have  $D\tilde \Theta^n(\varphi,
 \psi)=\Theta^n(\varphi,\Div(\psi\, b^{\e_n}))$ which is defined as $\int_M\varphi(x)
 \Div(\psi \, b^{\e_n}g(\cdot)(x))\,dx$. Since $b$ is time-dependent and not smooth we have
  to extend this definition. We let $(\tilde\a^i)_{i\ge 1}$ be a family of smooth functions 
  $[0,T]\to \RR$ such that, { possibly by extending the family $ \Psi^k$ defined in the 
  proof of  Theorem~\ref{T1},}
   linear combinations of functions $(t,x)\mapsto 
  \tilde\a^i(t)\Psi^k(x)$ with rational coefficients are dense in $L^q([0,T]\times M)$. 
  Now for all $\b(t,x)=\sum_{\ell=1}^La_\ell\tilde\a^{i_\ell}(t)\Psi^{k_\ell}(x)$ with 
  rational $a_\ell$ we can assume that the processes $$t\mapsto 
  \int_0^t\Theta_s^n( \Phi^j,\beta(s,
  \cdot))\,ds=\sum_{\ell=1}^La_\ell\int_0^t\Theta_s^n( \Phi^j,
  \Psi^{k_\ell})\tilde\a^{i_\ell}(s)\,ds 
 $$
all together converge in law as $n\to \infty$ to  processes $\di t\mapsto \int_0^t\Theta_s( \Phi^j,\beta(s,\cdot))\,ds$ which satisfy
\begin{equation}
\label{phibeta}
\EE\left[\int_0^T\Theta_s( \Phi^j,\beta(s,\cdot))^q\,ds\right]\le \| \Phi^j\|_q^p\|\beta\|_{L^q([0,T]\times M)}^q,
\end{equation}
from the fact that for all $n\ge 1$
\begin{equation}
\label{phibetabis}
\EE\left[\int_0^T\Theta_s^n( \Phi^j,\beta(s,\cdot))^q\,ds\right]\le \| \Phi^j\|_p^q\|\beta\|_{L^q([0,T]\times M)}^q.
\end{equation}
Here $p$ is the exponent conjugate to $q$: $\di p=\f{q}{q-1}$.

 This bound allows to define $\di t\mapsto \int_0^t\Theta_s(\varphi,\beta(s,\cdot))\,ds$ for all smooth $\varphi$ and $\beta\in L^q([0,T]\times M)$.  So $t\mapsto\int_0^t\Theta_s(\varphi,\Div(\psi \, b))\,ds$ is well defined (recall that $\Div(\psi b)=\langle d\psi,b\rangle$), and
  by an argument similar to the proof of Theorem~\ref{T1}  we see that $t\mapsto\int_0^t\Theta_s^n(\varphi,\Div(\psi \, b^{\e_n}))\,ds$ converges in law to $t\mapsto\int_0^t\Theta_s(\varphi,\Div(\psi \, b))\,ds$.
So we can make the identification $\int_0^tD\tilde \Theta_s(\varphi,\psi)\,ds=\int_0^t\Theta_s(\varphi,\Div(\psi \, b))\,ds$.

We are left to prove the bound for the kinetic energy. But this is exactly similar to the proof of Theorem~\ref{2var}.
\end{proof}

\section{Constructing generalized flows from solutions of finite variation transport equations}\label{Section6}
\setcounter{equation}0

In this section we aim to give an alternative construction of generalized flow with prescribed drift, using Ocone Pardoux method and weak solution of transport equations (in the sense of DiPerna and Lions).

To start with, let us consider a semimartingale flow $g(t)(x)$ satisfying $\  g(0)(x)\equiv x$ and
\begin{equation}
 \label{E46}
dg(t)(x)=\s(g(t)(x))\circ dW_t+b(t,g(t)(x),\om) dt,
\end{equation}
with the same assumptions as in the beginning of section~\ref{Section2}. In particular the vector fields $\s_i$ are divergence free. 
Assume that $\s$ and $b$ are $C^1$ in the space variable. Let $\tilde g(t)(x)$ be the martingale flow satisfying 
\begin{equation}
 \label{E47}
d\tilde g(t)(x)=\s(\tilde g(t)(x))\circ dW_t,\qquad \qquad \tilde g(0)(x)\equiv x. 
\end{equation}
Notice that $\tilde g(t)$ is measure preserving.
The method of Ocone and Pardoux (\cite{Ocone-Pardoux:89}) consists in writing 
\begin{equation}
 \label{E48}
g(t)(x)=\tilde g(t)\left(\psi(t)(x)\right)
\end{equation}
with $\psi(t)(x)$ a bounded variation flow to be determined. From~\eqref{E48} we get 
\begin{equation}
 \label{E49}
dg(t)(x)=(d\tilde g(t))\left(\psi(t)(x)\right)+T_{\psi(t)}\tilde g(t)\left(d\psi(t)(x)\right)
\end{equation}
where $T_{\psi(t)}\tilde g(t)$ is the tangent map at $\psi(t)$ to $x\mapsto \tilde g(t)(x)$ (we omit the dependence in~$\om$). This
 together with~\eqref{E46} and~\eqref{E47} yields 
\begin{equation}
 \label{E50}
d\psi(t)(x)=\tilde b(t,\psi(t)(x),\om)\,dt
\end{equation}
with 
\begin{equation}
 \label{E51}
\tilde b(t,y,\om)=\left(T_y\tilde g(t)(\cdot)\right)^{-1}\left(b(t,\tilde g(t)(y),\om)\right).
\end{equation}
For $\varphi\in C^\infty(M)$ define $\theta_t^{g,\varphi}$, $\theta_t^{\tilde g,\varphi}$, $\theta_t^{\psi,\varphi}$ as: for $\phi\in C^\infty(M)$
\begin{equation}
 \label{E53}
\Theta^g(\varphi, \phi)=\left(\theta_t^{g,\varphi}, \phi\right)_{L^2(M)},\ \Theta^{\tilde g}(\varphi, \phi)=\left(\theta_t^{\tilde g,\varphi}, \phi\right)_{L^2(M)},\Theta^\psi(\varphi, \phi)=\left(\theta_t^{\psi,\varphi}, \phi\right)_{L^2(M)}.
\end{equation}
From~\eqref{E11} and~\eqref{E48} we get 
\begin{equation}
 \label{E52}
\begin{split}
 \theta_t^{g,\varphi}&=\varphi\circ g(t)^{-1}\\
&=\varphi\circ \psi(t)^{-1}\circ \tilde g(t)^{-1}.
\end{split}
\end{equation}
and this yields
\begin{equation}
 \label{E54}
\Theta_t^g(\varphi,\phi)=\Theta_t^{\tilde g(t)}\left(\theta_t^{\psi,\varphi}, \phi\right)
\end{equation}
which implies 
\begin{equation}
 \label{E55}
\Theta_t^g(\varphi,\phi)=\Theta_t^\psi(\varphi,\phi\circ \tilde g(t))
\end{equation}
where we used  the fact that $\tilde g(t)$ is measure preserving.

\begin{lemma}

 \label{L1}
We have for all $\phi\in C^\infty(M)$, $t\in [0,T]$, a.s.
\begin{equation}
 \label{E56}
\int_M\left(\Div\tilde b(t,\cdot,\om)\right)(x)\phi(x)\,dx=\int_M\Div b(t,\cdot,\om)(y)\left(\phi\circ (\tilde g(t))^{-1}\right)(y)\,dy.
\end{equation}
In particular, if $\Div b(t,\cdot,\om)\equiv 0$ then $\Div \tilde b(t,\cdot,\om)\equiv 0$.
\end{lemma}
\begin{proof}
We will write $b(y)=b(t,y,\om)$,  $\tilde b(x)=\tilde b(t,x,\om)$
For $\phi\in C^\infty(M)$,
 \begin{align*}
  \int_M(\Div\tilde b)\phi&=\int_M\langle d\phi,\tilde b\rangle\\
&=\int_M\left\langle d\phi, (T\tilde g)^{-1}\circ b\circ \tilde g\right\rangle\\
&=\int_M \left\langle d\left(\phi\circ (\tilde g)^{-1}\right)(\tilde g(x)),  b(\tilde g(x))\right\rangle\,dx\\
&=\int_M \left\langle d\left(\phi\circ (\tilde g)^{-1}\right)(y),  b(y)\right\rangle\,dy\\
&=\int_M\left(\phi\circ (\tilde g)^{-1}\right)(y)(\Div b)(y)\,dy
 \end{align*}
where we used in the fourth equality the fact that $\tilde g$ is measure preserving.
\end{proof}

Now consider a deterministic drift $ b(t,x)$ such that $b\in L^1([0,T], L^q (TM))$ for some $q>1$, and $\Div b\equiv 0$ in the weak sense. It is easily seen that Lemma~\ref{L1} is still valid for $b$, so we have a.s. $\Div \tilde b\equiv 0$. Moreover a.s. $\tilde b\in L^1([0,T], L^q (TM))$. Under this condition we can apply Proposition~II.1 in \cite{DiPerna-Lions:89} and we deduce that a.s. the transport equation which is the weak version of~\eqref{E50}, namely
\begin{equation}
 \label{E57}
\f{\partial \theta_t}{\partial t}=-(\tilde b\cdot \n)\theta_t,\qquad \theta_0=\varphi, \qquad \varphi\in C^\infty(M),
\end{equation}
has a solution $\theta_t^{\tilde b,\varphi}$ in $\di \bigcap_{p\ge 1} L^\infty(0,T;L^p(M))$. Moreover since $\tilde b$ is adapted the process   $\theta_t^{\tilde b,\varphi}$ can also be chosen so that all $\theta_t^{\tilde b,\varphi}, \varphi\in C^\infty(M)$ and also $\tilde g, W$ are jointly adapted to the same filtration for which $W$ is still a cylindrical Brownian motion and $\tilde g$ satisfies~\eqref{E47}: instead of considering $\theta_t^{\tilde b,\varphi}$ as the limit in law of some regularized processes $\theta_t^{\tilde b^\e,\varphi}$, consider $((\theta_t^{\tilde b,\tilde\varphi^j},\tilde g, W)$ as limit in law of $((\theta_t^{\tilde b^\e,\tilde\varphi^j},\tilde g, W)$ for $(\tilde\varphi^j)_{j\ge 1}$ a dense subsequence in $C^\infty(M)$.

 So by analogy to~\eqref{E54} we define 
\begin{equation}
 \label{E58}
\Theta_t^{\s,b}(\varphi,\phi)=\Theta_t^{\tilde g}\left(\theta_t^{\tilde b,\varphi},\phi\right),\qquad \varphi,\phi\in C^\infty(M)
\end{equation}

\begin{prop}
 \label{P2}
 Take a deterministic drift $ b(t,x)$ such that $b\in L^1([0,T], L^q (TM))$ for some $q>1$, and $\Div b\equiv 0$ in the weak sense. Then the generalized flow 
$\Theta^{\s,b}$ defined in equation~\eqref{E58} is a generalized flow with kinetic energy
\begin{equation}
 \label{E62}
\SE_q'\left(\Theta^{\s,b}\right)\le\SE_q(b)
\end{equation}
where
\begin{equation}
 \label{E63}
\SE_q(b)=\f12\int_0^Tdt\int_Mdx\|b(t,x)\|^q
\end{equation}

\end{prop}

\begin{proof}
 We have 
\begin{equation}
\label{E59}
\begin{split}
\Theta^{\s,b}(\varphi,\phi)=\Theta_t^{\tilde g}\left(\theta_t^{\tilde b,\varphi},\phi\right)&=
\int_M\theta_t^{\tilde b,\varphi}\left((\tilde g(t))^{-1}(x)\right)\phi(x)\,dx\\
&=\int_M\theta_t^{\tilde b,\varphi}\left(x\right)\phi\left(\tilde g(t)(x)\right)\,dx
\end{split}
\end{equation}
and this implies
\begin{equation}
\label{E60}
\begin{split}
&d\tilde\Theta^{\s,b}(\varphi,\phi)\\&=\int_M\theta_t^{\tilde b,\varphi}\left(x\right)\left\langle d\phi,\ d^{\hbox{\sevenrm It\^o}}\tilde g(t)(x)\right\rangle\,dx+\int_M\theta_t^{\tilde b,\varphi}\left(x\right)\Div \left(\tilde b (\phi\circ\tilde g(t))\right)(x)\,dx\,dt\\
&=\int_M\theta_t^{\tilde b,\varphi}\left(x\right)\left\langle d\phi,\s\left(\tilde g(t)(x)\right) dW_t\right\rangle\,dx+\int_M\theta_t^{\tilde b,\varphi}\left(x\right)\left\langle d (\phi\circ\tilde g(t)), \ \tilde b\right\rangle(x)\,dx\,dt\\
&=\sum_{i\ge 1}\int_M\theta_t^{\tilde b,\varphi}\left(x\right)\left\langle d\phi,\s_i\right\rangle\left(\tilde g(t)(x)\right)\,dx\,dW_t^i+\int_M\theta_t^{\tilde b,\varphi}\left(x\right)\left\langle d \phi, b\right\rangle(\tilde g(t)(x))\,dx\,dt\\
&= \sum_{i\ge 1}\int_M\theta_t^{\tilde b,\varphi}\left((\tilde g(t))^{-1}(y)\right)\left\langle d \phi, \s_i\right\rangle(y)\,dydW_t^i+\int_M\theta_t^{\tilde b,\varphi}\left((\tilde g(t))^{-1}(y)\right)\left\langle d \phi, b\right\rangle(y)\,dy\,dt\\
&= \sum_{i\ge 1}\Theta^{\s,b}(\varphi, \langle d\phi,\s_i\rangle )\,dW_t^i+\Theta^{\s,b}(\varphi, \langle d\phi,b\rangle )\,dt
\end{split}
\end{equation}
where the first term in the right is the martingale part and the second term is the finite variation part. 
We prefer to write the last equality as 
\begin{equation}
\label{E61}
d\tilde\Theta^{\s,b}(\varphi,\phi)=\sum_{i\ge 1}\Theta^{\tilde g}\left(\theta_t^{\tilde b,\varphi}, \langle d\phi,\s_i\rangle \right)\,dW_t^i+\Theta^{\tilde g}\left(\theta_t^{\tilde b,\varphi}, \langle d\phi,b\rangle \right)\,dt.
\end{equation}
>From this equation the properties of a generalized flow are easily checked.

We are left to prove that $\SE_q'(\Theta^{\s,b})\le \SE_q(b)$. Again this can be done via a regularization procedure of $\tilde b$ of the form $\tilde b^\e=\left(e^{\e\square}(\tilde b^\flat)\right)^\sharp$, an extraction of subsequence, and similar estimates as before. We leave the details to the reader.

\end{proof}

\medbreak\noindent{\bf Acknowledgements}

The research of A.Antoniouk was supported by the grant no. 01-01-12
of National Academy of Sciences of Ukraine (under the joint
Ukrainian-Russian project of NAS of Ukraine and Russian
Foundation of Basic Research)."

This research was also supported by the French ANR grant ANR-09-BLAN-0364-01 ProbaGeo
   and by the Portuguese FCT project PTDC/MAT/104173/2008.\rm

\end{document}